\documentclass[12pt,a4paper]{amsart}

\usepackage{amssymb}
\usepackage{stmaryrd}
\usepackage{fullpage}

\usepackage[usenames,dvipsnames]{color}
\usepackage{pdfcolmk}
\usepackage{hyperref}

\usepackage[all]{xy}
\usepackage{pb-diagram,pb-xy}

\title{The observable structure of persistence modules}

\author{Fr\'{e}d\'{e}ric Chazal}
\address[Fr\'{e}d\'{e}ric Chazal]{GEOMETRICA, INRIA Saclay --- \^{I}le-de-France, 2--4, rue Jacques Monod, Orsay, 91893 Cedex, France}
\email{frederic.chazal@inria.fr}

\author{William Crawley-Boevey}
\address[William Crawley-Boevey]{Department of Pure Mathematics, University of Leeds, Leeds, LS2 9JT, UK}
\email{w.crawley-boevey@leeds.ac.uk}

\author{Vin de Silva}
\address[Vin de Silva]{Department of Mathematics, Pomona College, 610 N College Ave, Claremont, California, USA}
\email{vin.desilva@pomona.edu}


\theoremstyle{plain}
\newtheorem{thm}{Theorem}[section]
\newtheorem{lem}[thm]{Lemma}
\newtheorem{pro}[thm]{Proposition}
\newtheorem{cor}[thm]{Corollary}

\theoremstyle{definition}
\newtheorem{dfn}[thm]{Definition}

\newtheorem{exa}[thm]{Example}
\newtheorem{rem}[thm]{Remark}

\newcommand{\R}{\mathbb{R}}
\newcommand{\Q}{\mathbb{Q}}
\newcommand{\Z}{\mathbb{Z}}

\newcommand{\Pers}{\mathbf{Pers}}
\newcommand{\Eph}{\mathbf{Eph}}
\newcommand{\Ob}{\mathbf{Ob}}

\DeclareMathOperator{\Ker}{Ker}
\DeclareMathOperator{\Coker}{Coker}
\DeclareMathOperator{\Ima}{Im}

\DeclareMathOperator{\rank}{rank}

\DeclareMathOperator{\rad}{rad}

\DeclareMathOperator{\Hgr}{H}

\newcommand{\kk}{\mathbf{k}}
\DeclareMathOperator{\rk}{rk}
\DeclareMathOperator{\mult}{m}
\newcommand{\colim}{\operatorname{colim}}
\newcommand{\wto}{\dashrightarrow}

\newcommand{\inter}{\mathrm{d_i}}
\newcommand{\bottle}{\mathrm{d_b}}
\newcommand{\ellinf}{\mathrm{d}^\infty}
\newcommand{\dgm}{\mathrm{dgm}}

\newenvironment{vlist}%
 { \begin{list}%
         {$\bullet$}%
         {\setlength{\labelwidth}{20pt}%
          \setlength{\leftmargin}{25pt}%
          \setlength{\topsep}{1ex}
          \setlength{\itemsep}{1.5ex}
          \setlength{\parsep}{0pt}}}%
 { \end{list} }

\begin{document}

\begin{abstract}
In persistent topology, q-tame modules appear as a natural and large class of persistence modules indexed over the real line for which a persistence diagram is definable.  However, unlike persistence modules indexed over a totally ordered finite set or the natural numbers, such diagrams do not provide a complete invariant of q-tame modules.  The purpose of this paper is to show that the category of persistence modules can be adjusted to overcome this issue.  We introduce the \emph{observable category} of persistence modules: a localization of the usual category, in which the classical properties of q-tame modules still hold but where the persistence diagram is a complete isomorphism invariant and all q-tame modules admit an interval decomposition.
\end{abstract}

\maketitle

\section{Introduction}
\label{sec:intro}

\subsection{Discrete persistence modules}
\label{subsec:intro}

Topological persistence~\cite{Edelsbrunner_L_Z_2002,Zomorodian_Carlsson_2005} may be introduced with the observation that a nested sequence of topological spaces
\[
X_0 \subseteq X_1 \subseteq \dots \subseteq X_n
\]
gives rise to a sequence of vector spaces and linear maps
\[
\Hgr(X_0) \to \Hgr(X_1) \to \dots \to \Hgr(X_n)
\]
upon computing homology with coefficients in a field~$\kk$. In general, a diagram of vector spaces and linear maps
\[
V_0 \to V_1 \to \dots \to V_n
\]
is called a \emph{persistence module} indexed by $\{0,1,\dots,n\}$.
Such a diagram may be expressed as a direct sum of certain indecomposable diagrams called \emph{interval modules}~\cite{Zomorodian_Carlsson_2005}, parametrized by intervals $[p,q] \subseteq \{0, 1, \dots, n\}$. The interval module $V = \kk_I$ associated to an interval~$I$ is defined by
\[
V_i =
\begin{cases}
\kk \quad
&\text{if $i \in I$}
\\
0
&\text{otherwise}
\end{cases}
\]
with the nonzero maps $\kk \to \kk$ set equal to~1.

The number of direct summands $\mult_{p,q}$ of each type $\kk_{[p,q]}$ is independent of the specific decomposition, by the Krull--Schmidt principle or more directly from the formula
\[
\mult_{p,q}
=
\rank(V_p \to V_q)
-\rank(V_p \to V_{q+1})
-\rank(V_{p-1} \to V_q)
+\rank(V_{p-1} \to V_{q+1})
.
\]
As a result, the collection of numbers $(\mult_{p,q} \mid 0 \leq p \leq q \leq n)$ is a complete invariant of the persistence module, and an invariant of the initial topological data. It is typically expressed as a \emph{barcode} or \emph{persistence diagram}~\cite{Edelsbrunner_L_Z_2002,Zomorodian_Carlsson_2005}.

\subsection{Persistence modules for the real line}
The purpose of this short paper is to address some issues that arise when attempting to follow the same thought process for persistence modules indexed by the real line. Here are the main points of divergence:
\begin{vlist}
\item
Not every persistence module is decomposable into interval modules.
\item
Nonetheless, there are classes of persistence module---q-tame~\cite{Chazal_dS_G_O_2012}, pointwise finite-dimensional~\cite{Lesnick_2011}---for which a persistence diagram is definable. Even so, not every q-tame persistence module is decomposable into interval modules.

\item
The persistence diagram is not a complete invariant. Two non-isomorphic q-tame persistence modules may have the same persistence diagram. This is true even if we use a more refined invariant, the decorated persistence diagram~\cite{Chazal_dS_G_O_2012}.
\end{vlist}
To be fair, there are effective ways of working around these problems~\cite{Chazal_dS_G_O_2012,Lesnick_2011}.
What we offer here is the suggestion that the awkwardness dissolves completely if we make a small adjustment to the category of persistence modules that we work in.

The adjustment is motivated by the following principle: whereas persistence modules carry information at many different scales simultaneously, what matters most is how the information persists across scales (through the structure maps). Features that exist over a short range are regarded as relatively unimportant. In topological data analysis, such short-term information arises from noisy sampling, for instance. In the extreme case are \emph{ephemeral} features: non-zero vectors that are seen at exactly one index value. It is conventional to regard such information as statistically meaningless.

Our proposal is to build this principle---that of ignoring ephemeral information---directly into the category that we work with. The mechanism for doing so is Serre localization. The resulting \emph{observable category} of persistence modules turns out to be beautifully behaved. Persistence modules $\kk_{[p,q]}$, $\kk_{[p,q)}$, $\kk_{(p,q]}$, $\kk_{(p,q)}$ associated to different intervals with the same endpoints become isomorphic. Every q-tame module has an interval decomposition. The persistence diagram is a complete isomorphism invariant for q-tame modules.
And there is a very clean description of the morphisms in this category (something that is not always available for such constructions).

\subsection{Basic definitions}
\label{subsec:poset}

Let $(R,\leq)$ be a partially ordered set. The category $\Pers$ of persistence modules over~$R$ (or `indexed by~$R$') is defined as follows:
\begin{vlist}
\item
A \emph{persistence module} $V$ is a functor from $R$, considered in the natural way as a category,
to the category of vector spaces. Thus it consists of vector spaces $V_t$ for $t\in R$ and
linear maps $\rho_{ts}:V_s\to V_t$ for $s\le t$ called \emph{structure maps}, which satisfy $\rho_{ts} = \rho_{tu} \rho_{us}$ for all $s \le u\le t$
and $\rho_{tt} = 1_{V_t}$ for all~$t$.

\item
A \emph{morphism} $\phi: V \to W$ is a natural transformation between functors.
Thus, it is a collection of linear maps $\phi_t: V_t \to W_t$ such that $\phi_t \rho_{ts} = \sigma_{ts} \phi_s$ for all $s \leq t$.
(The maps $\sigma_{ts}$ are the structure maps for~$W$.) 
\end{vlist}

\begin{rem}
\label{rem:morphism2}
Equivalently, a morphism $\phi: V \to W$ is a collection of linear maps $\phi_{ts}: V_s \to W_t$ defined for $s \leq t$, such that $\phi_{ts} = \sigma_{tv} \phi_{vu} \rho_{us}$ whenever $s \leq u \leq v \leq t$.
The translation between the two formulations is given by $\phi_{st} = \phi_t \rho_{ts} = \sigma_{ts} \phi_s$ in one direction, and $\phi_t = \phi_{tt}$ in the other.
In what follows we will favour this second formulation.
\end{rem}

\begin{rem}
The category of persistence modules is an abelian category, and even a \emph{Grothendieck category}
meaning that it has a generator
and satisfies Grothendieck's (AB5) condition (see for example
\cite[Theorem 14.2]{Faith_1973}).
%
%
The {zero module}~$0$ is defined pointwise: it is the zero vector space at all index values~$t$.
Images, kernels, cokernels, exact sequences and direct sums are likewise defined pointwise.
\end{rem}

Our results only apply when the indexing set $R$ is totally ordered, so we assume this henceforth.
By an \emph{interval} in $R$ we mean a non-empty subset $I$ of $R$ with the property that
$s \le u \le t$ with $s,t\in I$ implies $u\in I$. 
The corresponding \emph{interval module} $V=\kk_I$ is given by $V_t = \kk$ for $t\in I$, $V_t=0$ for $t\notin I$, and $\rho_{ts}=1$ for $s,t\in I$ with $s\le t$.


\begin{exa}
Let $p < q$. Then we have closed, half-open and open intervals
\begin{alignat*}{2}
[p,q] &= \{ t \in R \mid p \leq t \leq q\}
&\qquad
[p,q) &= \{ t \in R \mid p \leq t < q\}
\\
(p,q] &= \{ t \in R \mid p < t \leq q\}
&\qquad
(p,q) &= \{ t \in R \mid p < t < q\}
\end{alignat*}
with endpoints $p,q \in R$. Not all intervals need be of this type (for instance if $R = \mathbb{Q}$).
\end{exa}

\begin{lem}
\label{lem:int-indecomp}
Interval modules are indecomposable: they cannot be expressed as a nontrival direct sum of submodules.
\end{lem}

\begin{proof}
The endomorphism ring of an interval module is isomorphic to~$\kk$. Indeed, for any endomorphism $\phi = (\phi_{ts})$ the non-trivial terms (those with $s,t \in I$) are scalars and, indeed, must be equal to the same scalar.
The projection maps in a direct-sum decomposition would be idempotent endomorphisms, but $\kk$ has no nontrivial idempotents.
\end{proof}

With this in mind, the natural question is whether every persistence module over a total order~$R$ decomposes as a direct sum of interval modules. The answer is yes when $R$ is finite or the natural numbers~\cite{Webb_1985}; and also yes in the special case of modules which are finite-dimensional at each index, assuming that $R$ is has a countable subset 
which is dense in the order topology on $R$~\cite{CrawleyBoevey_2012}. But in general 
there are persistence modules which do not decompose into intervals. 
Examples due to Webb, Lesnick and Crawley-Boevey can be found in~\cite{Chazal_dS_G_O_2012}.

\subsection{q-tame persistence modules}
Of particular importance in applications are the \emph{q-tame} persistence modules~\cite{Chazal_dS_G_O_2012}, which are defined by the condition that $\rank(\rho_{ts})$ be finite whenever $s < t$. Important examples indexed by the real line are:
\begin{vlist}
\item
Let $X$ be a locally compact polyhedron and let $f: X \to \R$ be a proper continuous map which is bounded below. Then $( \Hgr_*(f^{-1}(-\infty,t]))_{t \in \R}$ is q-tame. This includes the case where $f$ is the distance from a compact subset $A \subset \R^n$ in any norm.

\item
Let $X$ be a totally bounded metric space. Then the Vietoris--Rips and intrinsic \v{C}ech filtered complexes on~$X$ have q-tame persistent homology~\cite{Chazal_dS_O_2012}.
\end{vlist}
Many of these q-tame examples fail to be \emph{pointwise finite-dimensional}: there are index values where $\dim(V_t)$ is infinite.
For an extreme case, Droz~\cite{Droz_2012} has constructed a compact metric space whose Vietoris--Rips homology is uncountably infinite-dimensional at all values of~$t$ in an interval of positive length.

Our main results are summarized in the following theorem, which collates Corollaries~\ref{c:obisquo} and \ref{cor:main}, Theorem~\ref{thm:KSA}, Example~\ref{e:qtobs} and Propositions~\ref{pro:distinctness},
\ref{pro:inter-ob} and \ref{pro:dgm-ob}.
We state it here only for $R=\R$, but some parts hold more generally.


\begin{thm}
\label{thm:intro-main}
There is a quotient category $\Ob$ of the category of persistence modules over $\R$, with the following properties.
\begin{enumerate}
\item[(i)]
The property of a persistence module being q-tame, its undecorated diagram and the interleaving
distance between two persistence modules depend only on their images in $\Ob$.
\item[(ii)]
Any q-tame persistence module decomposes as a direct sum of interval modules in $\Ob$. The list of summands is essentially unique, and is determined by the persistence diagram.
\end{enumerate}
\end{thm}

The rest of the paper is organised as follows. In Section~\ref{sec:ob} we define and study the `observable' category $\Ob$. In Section~\ref{sec:decomposition} we study interval decompositions. In Section~\ref{sec:real} we apply our results to the motivating case of persistence modules over the real line.

\section{The Observable Category}
\label{sec:ob}

For this section we make the standing assumption that $(R,\leq)$ is a total order that is \emph{dense}: for every $s < t$ there exists an intermediate element $s < u < t$.

\subsection{Ephemeral modules}
\label{subsec:eph}

Following~\cite{Chazal_dS_G_O_2012}, we say that a persistence module is \emph{ephemeral} if $\rho_{ts} = 0$ whenever $s < t$. Let $\Eph$ denote the full subcategory of $\Pers$ whose objects are the ephemeral modules. The \emph{observable category} $\Ob$ will be equivalent to the quotient category of the category of persistence modules `modulo' ephemeral modules, following Serre's theory of localization~\cite{Serre_1953,Gabriel_Zisman_1967}. 

To this end, a morphism $\phi$ between persistence modules is called a \emph{weak isomorphism} if $\Ker\phi$ and $\Coker\phi$ are both ephemeral. Then the quotient category is obtained from $\Pers$ by inverting all the weak isomorphisms.

\begin{lem}
\label{lem:Serre}
The full subcategory of ephemeral modules satisfies the condition of Serre: given a short exact sequence of persistence modules
\[
\begin{diagram}
\node{0}
	\arrow{e}
\node{V'}
	\arrow{e,t}{\iota}
\node{V}
	\arrow{e,t}{\pi}
\node{V''}
	\arrow{e}
\node{0}
\end{diagram}
\]
either statement
\begin{enumerate}
\item
$V$ is ephemeral
\item
$V'$ and $V''$ are both ephemeral
\end{enumerate}
implies the other.
\end{lem}

The lemma ensures that the class of weak isomorphisms is closed under composition, thanks to the exact sequence
\[
\begin{diagram}
{0}
\longrightarrow
{\Ker \phi}
\longrightarrow
{\Ker \psi\phi}
\longrightarrow
{\Ker \psi}
\longrightarrow
{\Coker \phi}
\longrightarrow
{\Coker \psi\phi}
\longrightarrow
{\Coker \psi}
\longrightarrow
{0}
\end{diagram}
\]
for a pair of composable maps $V \stackrel{\phi}{\longrightarrow} V' \stackrel{\psi}{\longrightarrow} V''$.

\begin{proof}
%
If $V$ is ephemeral, then clearly so are $V'$ and $V''$. 
Conversely, suppose $V'$ and $V''$ are ephemeral and $s<t$. Choose $u$ with $s<u<t$,
and consider the diagram
\[
\begin{diagram}
\node[2]{V'_t}	
	\arrow{e,t}{\iota_t}
\node{V_t}	
\\
\node{0}
	\arrow{e,t}{}
\node{V'_u}	
	\arrow{e,t}{\iota_u}
	\arrow{n,l}{\rho'_{tu} = 0}
\node{V_u}	
	\arrow{e,t}{\pi_u}
	\arrow{n,r}{\rho_{tu}}
\node{V''_u}	
	\arrow{e,t}{}
\node{0}
\\
\node[3]{V_s}	
	\arrow{e,t}{\pi_s}
	\arrow{n,r}{\rho_{us}}
	\arrow{nw,b,..}{\alpha}
\node{V''_s}	
	\arrow{n,r}{\rho''_{us} = 0}
\end{diagram}
\]
Since $\pi_u \rho_{us} = \rho''_{us} \pi_s = 0$ and the middle row is exact, 
there is a map~$\alpha$ with $\rho_{us} = \iota_u \alpha$. 
Then $\rho_{ts} = \rho_{tu} \rho_{us} = \rho_{tu} \iota_u \alpha = \iota_t \rho'_{tu} \alpha = 0$.
Thus $V$ is ephemeral.
\end{proof}

We point out that if the total order is not dense then the ephemeral subcategory is not Serre. For $s < t$ with no intermediate element, the sets $\{s\}$, $\{t\}$ and $\{s,t\}$ are intervals. The short exact sequence of interval modules
\[
\begin{diagram}
\node{0}
	\arrow{e}
\node{\kk_{\{t\}}}
	\arrow{e}
\node{\kk_{\{s,t\}}}
	\arrow{e}
\node{\kk_{\{s\}}}
	\arrow{e}
\node{0}
\end{diagram}
\]
has ephemeral outer terms and a non-ephemeral middle term.

\subsection{Observable morphisms}
\label{subsec:ob-morphisms}

The quotient $\Pers \stackrel{\pi}{\to} \Pers/\Eph$ that we wish to construct
is characterized by the following universal property~\cite{Gabriel_Zisman_1967}: first, the functor $\pi$ carries ob-isomorphisms to isomorphisms; second, any other functor $\Pers \to \mathbf{C}$ that carries ob-isomorphisms to isomorphisms factorizes uniquely through~$\pi$.

Our plan is to define a category $\Ob$ and a functor $\Pers \stackrel{\pi}{\to} \Ob$ explicitly, and then verify the universal property. In this way, $\Ob = \Pers/\Eph$ (where `=' means `is a category equivalent to').

\begin{dfn}
\label{dfn:ob-morphism}
An \emph{observable morphism} (or \emph{ob-morphism}) of persistence modules $\phi^\circ: V \wto W$ is a collection of linear maps $\phi_{ts}: V_s \to W_t$ defined for $s < t$ (strictly less than), such that $\phi_{ts} = \sigma_{tv} \phi_{vu} \rho_{us}$ whenever $s \leq u < v  \leq t$.
Composition of ob-morphisms is defined as follows, 
using the fact that the index set $R$ is a dense order.
If $\phi^\circ:V\wto W$ and $\psi^\circ: W \wto X$ are observable morphisms, then we define $(\psi^\circ \phi^\circ)_{ts} =\psi_{tu}\phi_{us}$ for any~$u$ with $s<u<t$.
This is well-defined since if $s<u<v<t$ then
$\psi_{tv}\phi_{vs} = \psi_{tv} \sigma_{vu}\phi_{us} = \psi_{tu}\phi_{us}$.
Every persistence module~$V$ has an ob-identity $1^\circ_V = ( \rho_{ts} \mid s < t )$ extracted from its structure maps.
\end{dfn}

\begin{dfn}
The category of persistence modules and ob-morphisms is called the \emph{observable category} of persistence modules, $\Ob$.
It comes with a functor $\Pers \stackrel{\pi}{\to} \Ob$ which keeps the objects the same and maps each morphism $\phi = (\phi_{ts} \mid s \leq t)$ to an ob-morphism $\pi(\phi) = \phi^\circ = (\phi_{ts} \mid s < t)$ by forgetting the terms $\phi_{tt}$.
\end{dfn}

\begin{exa}
\label{ex:interval-maps}
Between every ordered pair of the four interval modules $\kk_{(p,q)}$, $\kk_{[p,q)}$, $\kk_{(p,q]}$ and $\kk_{[p,q]}$ there is a nonzero ob-morphism defined by setting $\phi_{ts} = 1$ wherever domain and range both equal~$\kk$. It follows that the four interval modules are isomorphic in~$\Ob$. This contrasts with the situation in~$\Pers$ where nonzero maps exist only between certain pairs. The situation is summarized as follows:
\[
\quad
{\xymatrix@=3pc{
\kk_{(p,q]}
	\ar@{<-->}[r]
	\ar@{<-->}[d]
	\ar@{<-->}[dr]
&
\kk_{(p,q)}
	\ar@{<-->}[d]
	\ar@{<-->}[dl]
\\
\kk_{[p,q]}
	\ar@{<-->}[r]
&
\kk_{[p,q)}
}}
\quad
\quad
\quad
{\xymatrix@=3pc{
\kk_{(p,q]}
	\ar[r]
	\ar[d]
	\ar[dr]
&
\kk_{(p,q)}
	\ar[d]
\\
\kk_{[p,q]}
	\ar[r]
&
\kk_{[p,q)}
}}
\]
In general there is a nonzero ob-morphism $\kk_{I} \wto \kk_{J}$ if and only if $\inf(J) \leq \inf(I) < \sup(J) \leq \sup(I)$ (these limits being interpreted in the completion of~$R$).
\end{exa}

\begin{exa}
\label{exa:int-ob-indecomp}
For a non-singleton interval~$I$, the ob-endomorphism ring of the interval module $\kk_I$ is isomorphic to~$\kk$.
(The proof of Lemma~\ref{lem:int-indecomp} applies verbatim. The `non-singleton' condition guarantees that there is at least one non-trivial  $\phi_{ts}$.)
\end{exa}

\begin{exa}
If $V$ is ephemeral then $1^\circ_V = 0$ and therefore every ob-morphism to or from~$V$ is zero. Thus $V$ is zero (that is, both initial and terminal) in $\Ob$.
\end{exa}

In the remainder of this subsection we show $\Ob$ is equivalent to the localized category $\Pers/\Eph$, by establishing that $\Pers \stackrel{\pi}{\to} \Ob$ satisfies the universal property described above. Here is the first part of the universal property:

\begin{thm}
\label{thm:univ1}
If $\phi: V \to W$ is a weak isomorphism then $\phi^\circ$ is invertible in $\Ob$.
\end{thm}

\begin{proof}
We construct an inverse $\psi^\circ = (\psi_{st} \mid s < t)$ as follows. Given $s < t$, select an intermediate index~$u$.

Since $\Coker \phi$ is ephemeral, the composition of $\sigma_{us}:W_s\to W_u$ with the 
natural map $W_u\to \Coker\phi_{uu}$ is zero. Thus $\sigma_{us}$ factors as a map
$\omega_{us}:W_s\to \Ima\phi_{uu}$ followed by the inclusion of $\Ima\phi_{uu}$ into $W_u$.

Dually, since $\Ker\phi$ is ephemeral, the composition of the inclusion $\Ker\phi_{uu}\to V_u$
and $\rho_{tu}:V_u \to V_t$ is zero. Thus there is an induced morphism 
$\tau_{tu}:\Ima\phi_{uu} \to V_t$ whose composition with the natural map 
$V_u \to \Ima\phi_{uu}$ is $\rho_{tu}$. 

We define $\psi_{ts} = \tau_{tu}\omega_{us}$.
%
%
%
%
%
%
It is straightforward to verify that this construction does not depend on the choice of intermediate element~$u$, and that it defines an ob-morphism $\psi^\circ: W \wto V$ that is inverse to $\phi^\circ: V \wto W$.
\end{proof}

\begin{dfn}
Let $V$ be a persistence module. Define a persistence module $\overline{V}$ by setting
\begin{alignat*}{2}
\overline{V}_t &=& \;\colim &( V_s \mid s < t )
\end{alignat*}
at each index~$t$. The structure maps $\bar\rho_{ts}$ are defined using the universal property of colimits.
The universal property also generates the following maps:

\begin{vlist}
\item
A morphism $n^V: \overline{V} \to V$, induced by the maps $(\rho_{ts} \mid s < t)$.

\item
A morphism $\bar\phi: \overline{V} \to \overline{W}$ for every ob-morphism $\phi^\circ: V \wto W$.
\end{vlist}
This last operation respects composition and identities, so `bar' is a functor $\Ob \to \Pers$.
One can show that this is a left adjoint for~$\pi$.
\end{dfn}

\begin{pro}
Each $n^V: \overline{V} \to V$ is a weak isomorphism.
\end{pro}

\begin{proof}
For every $s < t$ we have a commutative diagram:
\[
\begin{diagram}
\node{\overline{V}_t}
	\arrow{e,t}{n^V_t}
\node{V_t}
\\
\node{\overline{V}_s}
	\arrow{e,b}{n^V_s}
	\arrow{n,l}{\bar\rho_{ts}}
\node{V_s}
	\arrow{n,r}{\rho_{ts}}
	\arrow{nw}
\end{diagram}
\]
From this we see that $\bar\rho_{ts}$ carries $\Ker (n^V_s)$ to zero, while $\rho_{ts}$ carries $V_s$ to $\Ima (n^V_t)$ and hence to zero in $\Coker(n^V_t)$. Thus $\Ker(n^V)$ and $\Coker(n^V)$ are ephemeral.
\end{proof}

\begin{rem}
Similarly, the functor~$\pi$ has a right adjoint defined on objects by $\underline{V}_t = \lim(V_u \mid u > t)$, and there is a weak isomorphism $u^V: V \to \underline{V}$.
\end{rem}

\begin{exa}
If $V = \kk_{(p,q)}$, $\kk_{[p,q)}$, $\kk_{(p,q]}$ or $\kk_{[p,q]}$ then $\overline{V} = \kk_{(p,q]}$ and $\underline{V} = \kk_{[p,q)}$. All five morphisms in Example~\ref{ex:interval-maps} are instances of $n^V$ or~$u^V$. They become invertible in~$\Ob$.
\end{exa}

Now we prove the second part of the universal property.

\begin{thm}
\label{thm:univ2}
If $F: \Pers \to \mathbf{C}$ is a functor that carries weak isomorphisms to isomorphisms, then there is a unique functor $G: \Ob \to \mathbf{C}$ such that $F = G \pi$.
\end{thm}

\begin{proof}
Since $\Ob$ has the same objects as $\Pers$, it follows that $G$ is uniquely defined and satisfies $F = G\pi$ on objects. It remains to consider morphisms.

Let $\phi^\circ: V \wto W$ be an ob-morphism. We have a mixed-category diagram~(ii)
\[
\text{(i)}
{\xymatrix@=3pc{
\overline{V}
	\ar[r]^{\bar\phi}
	\ar[d]_{n^V}
& \overline{W}	
	\ar[d]^{n^W}
\\V
	\ar[r]_{\phi}
&W
}}
\quad
\quad
\quad
\text{(ii)}
{\xymatrix@=3pc{
\overline{V}
	\ar[r]^{\bar\phi}
	\ar[d]_{n^V}
& \overline{W}	
	\ar[d]^{n^W}
\\V
	\ar@{-->}[r]_{\phi^\circ}
&W
}}
\quad
\quad
\text{(iii)}
{\xymatrix@=3pc{
\overline{V}
	\ar@{-->}[r]^{\pi(\bar\phi)}
	\ar@{-->}[d]_{\pi(n^V)}
& \overline{W}	
	\ar@{-->}[d]^{\pi(n^W)}
\\V
	\ar@{-->}[r]_{\phi^\circ}
&W
}}
\]
which commutes after applying $\pi$ to the top three morphisms (iii).
By assumption, $F(n^V)$ is invertible and we are forced to define
\[
G(\phi^\circ) = F(n_W)  F(\bar{\phi})  F(n_V)^{-1}.
\]
Since $\overline{\psi \phi} = \bar\psi \bar\phi$ it follows that~$G$, defined in this way, is indeed a functor.

Now suppose $\phi^\circ = \pi(\phi)$ for some morphism $\phi: V \to W$. Then we have a commutative diagram (i) in $\Pers$ to which we apply $F$ to get $F(\phi) F(n^V) = F(n^W) F(\bar\phi)$. Since $F(n^V)$ is invertible we deduce
\[
F(\phi) = F(n_W)  F(\bar{\phi})  F(n_V)^{-1} = G(\phi^\circ).
\]
Hence $F = G \pi$ on morphisms.
\end{proof}

Theorems \ref{thm:univ1} and~\ref{thm:univ2} have the following consequence.

\begin{cor}
\label{c:obisquo}
$\Ob = \Pers/\Eph$.
\qed
\end{cor}

\subsection{Observable invariants}
\label{subsec:ob-invariants}

Because there are more isomorphisms in the observable category, there are fewer isomorphism invariants. In this subsection we consider which quantities and constructions `make sense' in the observable category. A function on persistence modules is a \emph{strict invariant} if it is invariant under isomorphisms in~$\Pers$; it is an \emph{observable invariant} if it is is invariant under ob-isomorphisms.

\begin{exa}
Let $t \in R$. Then $\rk_t(V) = \dim(V_t)$ is a strict invariant but not an observable invariant of the persistence module~$V$.
\end{exa}

\begin{exa}
Let $s < t$. Then $\rk_{st}(V) = \rank(\rho_{ts}: V_s \to V_t)$ is a strict invariant but not an observable invariant of the persistence module~$V$.
\end{exa}

\begin{exa}
Let $s < t$. Then each of the four `limiting ranks'
\begin{alignat*}{2}
\rk_{[st]} (V) &= \rank(\overline{V}_s \to \underline{V}_t)
&\qquad
\rk_{[st)} (V) &= \rank(\overline{V}_s \to \overline{V}_t)
\\
\rk_{(st]} (V) &= \rank(\underline{V}_s \to \underline{V}_t)
&\qquad
\rk_{(st)} (V) &= \rank(\underline{V}_s \to \overline{V}_t)
\end{alignat*}
is an observable invariant. We have $\rk_{[st]} \leq \{ \rk_{st}, \rk_{[st)}, \rk_{(st]} \} \leq \rk_{(st)}$.
\end{exa}

\begin{proof}
The limiting ranks are observable because `bar' and `underbar' are functors $\Ob \to \Pers$. The factorization
\[
\begin{diagram}
\node{\overline{V}_s}
	\arrow{e}
\node{{V}_s}
	\arrow{e}
\node{\underline{V}_s}
	\arrow{e}
\node{\overline{V}_t}
	\arrow{e}
\node{{V}_t}
	\arrow{e}
\node{\underline{V}_t}
\end{diagram}
\qedhere
\]
implies the given inequalities.
\end{proof}

\begin{rem}
For a q-tame persistence module we have the following formul\ae:
\begin{alignat*}{2}
\rk_{[st]}(V)
&=
&\;\max &\left( \rk_{s^-t^+}(V) \mid s^- < s < t < t^+ \right)
\\
\rk_{(st)}(V)
&=
&\min &\left( \rk_{s^+ t^-}(V) \mid s < s^+ < t^- < t \right)
\end{alignat*}
\end{rem}

\begin{exa}
\label{e:qtobs}
The property of being q-tame is observable.
\end{exa}

\begin{proof}
Since $\rk_{s^-t^+} \leq \rk_{[st]} \leq \rk_{st}$ whenever $s^- < s < t < t^+$, it follows that $V$ is q-tame if and only if $\rk_{[st]}(V) < \infty$ whenever $s < t$. This criterion is observable.
\end{proof}

The \emph{order topology} on~$R$ has basis given by the \emph{basic open sets}
$(s,t) = \{x\in R : s<x<t \}$, $(s,\infty) = \{x\in R:s<x\}$, $(-\infty,t) = \{ x\in R : x<t\}$
and $(-\infty,\infty)=R$.
Note that any basic open set is an open interval, provided it is non-empty, but there may be others,
such as $\Q\cap (0,\sqrt{2})$ for $R=\Q$.
The \emph{interior} of any subset $X$ of $R$ is then the
union of all basic open sets contained in $X$.

The reader may easily verify the following lemma.

\begin{lem}
\label{lem:interiors}
In a dense total order, an interval has empty interior if and only if it is a singleton. If two intervals $I,J$ have the same non-empty interior, then that interior includes all basic open sets whose endpoints lie in $I \cup J$.
\qed
\end{lem}

\begin{pro}
\label{pro:distinctness}
Interval modules $\kk_I, \kk_J$ are observably isomorphic if and only if the intervals $I, J$ have the same interior.
\end{pro}

\begin{proof}
If the interiors of the intervals differ, then there is a basic open set $(s,t)$ contained in one of $I,J$ but not the other. Then $\rk_{(st)}(\kk_I) \ne \rk_{(st)}(\kk_J)$ so the interval modules are not ob-isomorphic.

Conversely, suppose $I,J$ have the same interior. If it is empty then $\kk_I, \kk_J$ are ephemeral and therefore ob-isomorphic.
Otherwise we define an ob-morphism $\phi: \kk_I \to \kk_J$ by setting $\phi_{ts} = 1$ whenever $s \in I$ and $t \in J$ (and 0 otherwise); we define the ob-inverse $\psi: \kk_J \to \kk_I$ symmetrically.
The equations
\begin{alignat*}{3}
&\phi_{ts} = \rho^J_{tv} \phi_{vu} \rho^I_{us},
\quad
&&\psi_{ts} = \rho^I_{tv} \psi_{vu} \rho^J_{us},
\qquad
&&\text{for $s \leq u < v \leq t$}
\\
&\rho^I_{ts} = \psi_{tu}\phi_{us},
\quad
&&\rho^J_{ts} = \phi_{tu}\psi_{us},
\qquad
&&\text{for $s < u< t$}
\end{alignat*}
follow from Lemma~\ref{lem:interiors}, which implies for each equation that if the left-hand side is equal to~$1$ then so is the right-hand side (the converse being obvious).
\end{proof}

\section{Interval Decomposition}
\label{sec:decomposition}

In this section $(R,\le)$ is a total order. Recall that $R$ is said to be a dense order
if for every $s < t$ there is an intermediate element $s < u < t$.
We say that an interval $I$ in $R$ is \emph{left separable} if it has a countable subset $S\subseteq I$
such that for all $t\in I$ there is $s\in S$ with $s\le t$. (It is equivalent that $I$ equipped
with the left order topology is a separable topological space.)
Clearly $\R$ is dense and any interval $I$ in $\R$ is left separable, so
all the results in this section apply for the real line.

\subsection{Decomposition of persistence modules with chain conditions}
\label{subsec:chain}

In this subsection we prove a mild generalization of the main result of \cite{CrawleyBoevey_2012}.
In the next subsection we apply it to \mbox{q-tame} persistence modules.

\begin{dfn}
Let $V$ be a persistence module over a total order~$R$.

(i) One says that $V$ has the \emph{descending chain condition on images} provided that for all $t,s_1,s_2,\dots\in R$
with $t \ge s_1 > s_2 > \dots$, the chain
\[
V_t \supseteq \Ima(\rho_{t s_1}) \supseteq \Ima(\rho_{t s_2}) \supseteq \dots
\]
stabilizes \cite{CrawleyBoevey_2012}.

(ii) Given $s,t\in R$ with $s\le t$, we say that $V_s$ has the
\emph{descending chain condition on $t$-bounded kernels} provided that
for all $r_1,r_2,\dots\in R$ with $t < \dots < r_2 < r_1$, the chain
\[
V_s \supseteq \Ker(\rho_{r_1 s}) \supseteq \Ker(\rho_{r_2 s}) \supseteq \dots
\]
stabilizes. Applying $\rho_{ts}$, it is equivalent that the chain
\[
\Ima(\rho_{ts}) \supseteq \Ima(\rho_{ts}) \cap \Ker(\rho_{r_1 s}) \supseteq \Ima(\rho_{ts}) \cap \Ker(\rho_{r_2 s}) \supseteq \dots
\]
stabilizes.

(iii) We say that $V$ has the \emph{descending chain condition on sufficient bounded kernels} provided that for all
$t\in R$ and $0\neq v\in V_t$, there exists $s\le t$
such that $v\in \Ima(\rho_{ts})$ and $V_s$ has the descending chain condition on $t$-bounded kernels.
\end{dfn}

Note that condition (iii) holds if $V$ has the descending chain condition on kernels, as considered in \cite{CrawleyBoevey_2012}, since
one can then take $s=t$. 
The following theorem thus generalizes \cite[Theorem 1.2]{CrawleyBoevey_2012}.

\begin{thm}
\label{t:main}
Suppose that $R$ is a total order with the property that any interval in $R$ is left separable.
Then any persistence module with the descending chain condition 
on images and on sufficient bounded kernels is a direct sum of interval modules.
\end{thm}

For the proof we freely use the notation and results of \cite{CrawleyBoevey_2012}.
The hypothesis in that paper that $R$ have a countable subset which is dense in the
order topology was only used in \cite[Lemma 3.2]{CrawleyBoevey_2012},
but it is stronger than is required (for example consider $\R^2$ with the
lexicographic ordering), so we have replaced it here with the left separability hypothesis on intervals.

Suppose that $V$ has the descending chain condition on images.
Of the results in \cite{CrawleyBoevey_2012}, Lemmas 2.1(a) and 2.2 hold, all results in sections 3--6 hold, and Lemma 7.1(a) holds.
What fails is Lemma 2.1(b). Then in Lemma 7.1(b) the set is disjoint, but needn't strongly cover $V_t$.
The following is a partial replacement for Lemma 2.1(b).

\begin{lem}
\label{l:bkercut}
Let $s\le t$ and suppose that $V_s$ has the descending chain condition on $t$-bounded kernels.
Suppose that $c$ is a cut with $t\in c^-$ and $c^+\neq\emptyset$.
Then $\Ima(\rho_{ts})\cap \Ker^+_{c t} = \Ima(\rho_{ts})\cap \Ker(\rho_{rt})$ for some $r\in c^+$.
\end{lem}

\begin{proof}
Suppose that $\Ima(\rho_{ts})\cap \Ker^+_{c t} \neq \Ima(\rho_{ts})\cap \Ker(\rho_{rt})$ for all $r\in c^+$.
Since $c^+$ is non-empty, we can choose $r_1\in c^+$.
Since $\Ima(\rho_{ts})\cap \Ker^+_{c t} \neq \Ima(\rho_{ts})\cap \Ker(\rho_{r_1 t})$
there must be some $r_2$ with $\Ima(\rho_{ts})\cap \Ker(\rho_{r_2 t})$
strictly contained in $\Ima(\rho_{ts})\cap \Ker(\rho_{r_1 t})$.
Similarly,
since $\Ima(\rho_{ts})\cap \Ker^+_{c t} \neq \Ima(\rho_{ts})\cap \Ker(\rho_{r_2 t})$,
there must be some $r_3$ with $\Ima(\rho_{ts})\cap \Ker(\rho_{r_3 t})$
strictly contained in $\Ima(\rho_{ts})\cap \Ker(\rho_{r_2 t})$, and so on.
But then the chain
\[
\Ima(\rho_{ts})\cap \Ker(\rho_{r_1 t}) \supset \Ima(\rho_{ts})\cap \Ker(\rho_{r_2 t}) \supset \Ima(\rho_{ts})\cap \Ker(\rho_{r_3 t})
\supset \dots
\]
doesn't stabilize.
\end{proof}

\begin{proof}[Proof of Theorem~\ref{t:main}]
Suppose that $V$ has the descending chain condition on images and on sufficient bounded kernels.
As in \cite[\S 5]{CrawleyBoevey_2012}, one obtains submodules $W_I$ of $V$.

For $t\in R$, as in the proof of \cite[Theorem 1.2]{CrawleyBoevey_2012} there are 
sections $(F_{It}^-,F_{It}^+)$ for $I$ an interval which contains $t$, where
\[
F^\pm_{It} = \Ima^-_{\ell t} + \Ker^\pm_{ut} \cap \Ima^+_{\ell t},
\]
satisfying 
$F_{It}^+ = W_{It}\oplus F_{It}^-$. These sections no longer need to cover $V_t$,
but they are still disjoint, so by the argument in \cite[Lemma 6.1]{CrawleyBoevey_2012}
the sum of the $W_{It}$ is a direct sum.

Thus we obtain a submodule $\bigoplus_I W_I$ of $V$.
By \cite[Lemma 5.3]{CrawleyBoevey_2012} this submodule is a direct sum of interval modules.
We need to show it is equal to $V$.
Assume for a contradiction that there is $t\in R$ and an element $v\in V_t$ not in $\bigoplus_I W_{It}$.
By assumption there is $s\le t$
such that $v\in \Ima(\rho_{ts})$ and $V_s$ has the descending chain condition on $t$-bounded kernels.

Let $X = (\bigoplus_I W_{It}) \cap \Ima(\rho_{ts})$.
Since $v\in \Ima(\rho_{ts})$ but $v\notin X$,
we have $\Ima(\rho_{ts}) \not\subseteq X$.
Thus by \cite[Lemma 7.1(a)]{CrawleyBoevey_2012} there is a cut $\ell$ with $t\in \ell^+$ and
\[
X + \Ima^-_{\ell t} \cap \Ima(\rho_{ts}) \neq X + \Ima^+_{\ell t} \cap \Ima(\rho_{ts}).
\]
This inequality can only happen if $\Ima(\rho_{t s}) \not\subseteq \Ima_{\ell t}^-$,
so $s\notin \ell^-$, and hence $s\in\ell^+$.
Thus $\Ima^+_{\ell t} \subseteq \Ima(\rho_{ts})$.
Thus the inequality simplifies to
\[
X + \Ima^-_{\ell t} \neq X + \Ima^+_{\ell t}.
\]
Let $Y = X + \Ima^-_{\ell t}$. Clearly $\Ima^+_{\ell t} \not\subseteq Y$. Define
\begin{align*}
u^- &= \{ r \in R : \text{$r<t$ or $r\ge t$ and $\Ker(\rho_{rt})\cap \Ima^+_{\ell t} \subseteq Y$} \}\text{, and}
\\
u^+ &= \{ r \in R : \text{$r\ge t$ and $\Ker(\rho_{rt})\cap \Ima^+_{\ell t} \not\subseteq Y$} \}.
\end{align*}
Then $u$ is a cut and $t\in u^-$.

Now $\Ker^-_{u t} \cap \Ima^+_{\ell t} \subseteq Y$ since
\[
\Ker^-_{u t} \cap \Ima^+_{\ell t} = \bigcup_{\substack{r\in u^- \\ t\le r}} \Ker(\rho_{rt}) \cap \Ima^+_{\ell t}
\]
and by the definition of $u^-$, each term in the union is contained in $Y$.
We show that $\Ker^+_{u t} \cap \Ima^+_{\ell t} \not\subseteq Y$.
This is clear if $u^+$ is empty, for then $\Ker^+_{u t} = V_t$.
Thus suppose that $u^+$ is non-empty.
Since $\Ima^+_{\ell t} \subseteq \Ima(\rho_{ts})$ we have
\[
\Ker^+_{u t} \cap \Ima^+_{\ell t} = \Ker^+_{u t} \cap \Ima(\rho_{ts}) \cap \Ima^+_{\ell t}.
\]
Since $V_s$ has the descending chain condition on $t$-bounded kernels,
by Lemma~\ref{l:bkercut} there is some $r\in u^+$, such that this is equal to
\[
\Ker(\rho_{rt}) \cap \Ima(\rho_{ts}) \cap \Ima^+_{\ell t} = \Ker(\rho_{rt}) \cap \Ima^+_{\ell t},
\]
and by the definition of $u^+$ we have $\Ker(\rho_{rt}) \cap \Ima^+_{\ell t} \not\subseteq Y$.

Now since $t\in u^-$ and $t\in \ell^+$, the cuts $u$ and $\ell$ define an interval $I$ which contains $t$.
As already observed, we have
\[
W_{It} \oplus (\Ima^-_{\ell t} + \Ker^-_{u t} \cap \Ima^+_{\ell t}) = \Ima^-_{\ell t} + \Ker^+_{u t} \cap \Ima^+_{\ell t}.
\]
It follows that $W_{It} \subseteq \Ima^+_{\ell t} \subseteq \Ima(\rho_{ts})$, so $W_{It} \subseteq X$.
Then
\[
\begin{array}{ccl}
Y &=& Y + \Ker^-_{u t} \cap \Ima^+_{\ell t}
\\
&=& X + \Ima^-_{\ell t} + \Ker^-_{u t} \cap \Ima^+_{\ell t}
\\
&=& X + W_{It} + \Ima^-_{\ell t} + \Ker^-_{u t} \cap \Ima^+_{\ell t}
\\
&=& X + \Ima^-_{\ell t} + \Ker^+_{u t} \cap \Ima^+_{\ell t}
\\
&=& Y +  \Ker^+_{u t} \cap \Ima^+_{\ell t},
\end{array}
\]
a contradiction.
Thus $V = \bigoplus_I W_{I}$.
\end{proof}

\subsection{Decomposition of q-tame modules}
\label{subsec:q-tame}

In this section we prove an interval decomposition theorem for \mbox{q-tame} persistence modules in the observable category for a total order which is dense and has the property that all intervals are left separable.

\begin{dfn}
The \emph{radical} of a persistence module $V$ is the submodule $\rad V$ of~$V$ defined by
\[
(\rad V)_t = \sum_{s<t} \Ima(\rho_{ts}).
\]
By construction, it is the smallest submodule of $V$ such that $(V/\rad V)$ is ephemeral.
We say that $V$ is \emph{radical} if $V = \rad V$.
\end{dfn}

Observe that if $V$ is a q-tame persistence module, 
then $V$ has the descending chain condition on images and $V_s$ has the 
descending chain condition on $t$-bounded kernels for all $s<t$. 
If in addition $V$ is radical, it follows that $V$ has the 
descending chain condition on sufficient bounded kernels. 
Thus Theorem~\ref{t:main} gives:

\begin{cor}
\label{c:qtamerad}
If every interval in $R$ is left separable, 
then any radical q-tame persistence module is a direct sum of interval modules.
\end{cor}

Now suppose that $R$ is a dense order. In this case
$\rad\rad V = \rad V$ for any $V$, so $\rad V$ is a radical persistence module. 
Clearly any submodule of a q-tame persistence module is again q-tame.
Thus we obtain:

\begin{cor}
\label{c:qtame}
Suppose $R$ is dense and every interval in $R$ is left separable.
If $V$ is a q-tame persistence module, then $\rad V$ is a direct sum of interval modules.
\end{cor}

\begin{exa}
If $R$ is the set of real numbers, the product of the interval modules associated to the intervals
$[-1/n,1/n]$ with $n\ge 1$ is q-tame, and its radical is the direct sum of the interval modules for the
intervals $(-1/n,1/n]$. Neither of these modules satisfies the descending chain condition on kernels, as in \cite{CrawleyBoevey_2012}.
\end{exa}

Suppose again that $R$ is a dense order.
Since the observable category $\Ob$ is identified with the quotient category $\Pers/\Eph$, and the functor $\pi : \Pers \to \Ob$ has a right adjoint, it follows that $\Eph$ is a localizing subcategory in the sense of \cite[p372]{Gabriel_1962}. Therefore $\Ob$ is a Grothendieck category by \cite[Proposition 9, p378]{Gabriel_1962} and $\pi$ commutes with direct sums.  Thus direct sums exist in $\Ob$, and are given in the same way as in $\Pers$: by taking the direct sum of the vector spaces for each point of~$R$.

For any persistence module $V$, the inclusion $\rad V \to V$ is a weak isomorphism.
(In fact, $\rad V$ is the image of the weak isomorphism $n^V: \overline{V} \to V$ from Section~\ref{subsec:ob-morphisms}). Thus we reach our main goal:

\begin{cor}
\label{cor:main}
Suppose $R$ is dense and every interval in $R$ is left separable.
If $V$ is a q-tame persistence module, 
then $V$ is isomorphic in~$\Ob$ to a direct sum of interval modules.
\end{cor}

This decomposition is in fact essentially unique.
There is a version of the Krull--Remak--Schmidt--Azumaya Theorem for Grothendieck categories, 
see \cite[\S 6.7]{Bucur_Deleanu_1968} or \cite[\S 4.8]{Pareigis_1970}. 
It says that if an object is written as a direct sum of objects in two different ways, 
and if each summand has local endomorphism ring, then the terms in the two sums 
can be paired off in such a way that corresponding summands are isomorphic.

In particular, since by Example~\ref{exa:int-ob-indecomp} interval modules (for non-singleton intervals) have ob-endomorphism ring equal to~$\kk$, which is a local ring, the Krull--Remak--Schmidt--Azumaya Theorem and Proposition~\ref{pro:distinctness} give the following result.

\begin{thm}
\label{thm:KSA}
Over a dense total order, if a persistence module is isomorphic in $\Ob$ to a direct sum of 
interval modules in two different ways, then the non-singleton intervals in each sum can 
be paired off in such a way that corresponding intervals have the same interior.
\qed
\end{thm}

\section{Real-Parameter Persistence Modules}
\label{sec:real}

We return to the motivating case of persistence modules indexed by~$\R$.

\subsection{Interleavings and diagrams}
\label{subsec:int-diag}

Persistence modules over the real line are codified and studied using their persistence diagrams. The principal results are the stability theorem~\cite{CohenSteiner_E_H_2007,Chazal_dS_G_O_2012} and Lesnick's isometry theorem~\cite{Lesnick_2011,Chazal_dS_G_O_2012}. We review these results now.

\medskip
Two persistence modules $V,W$ are compared by finding \emph{interleavings} between them. An \emph{$\epsilon$-interleaving} is specified by collections of maps $\phi_{ts}: V_s \to W_t$ and $\psi_{ts}: W_s \to V_t$, defined for $t \geq s + \epsilon$, such that the equations
\begin{alignat*}{3}
&\phi_{ts} = \sigma_{tv} \phi_{vu} \rho_{us},
\qquad
&&\psi_{ts} = \rho_{tv} \psi_{vu} \sigma_{us},
\\
&\rho_{ts} = \psi_{tu}\phi_{us},
\qquad
&&\sigma_{ts} = \phi_{tu}\psi_{us},
\end{alignat*}
are satisfied whenever they are defined. It is immediate that
\begin{vlist}
\item
an isomorphism is the same thing as a $0$-interleaving;

\item
an ob-isomorphism restricts to $\epsilon$-interleavings for all $\epsilon > 0$.
\end{vlist}
The interleaving distance between two persistence modules is defined thus:
\[
\inter(V,W)
=
\inf \left( \epsilon
\mid
\text{there exists an $\epsilon$-interleaving between $V,W$}
\right)
\]
It is an extended pseudometric, taking values in $[0,\infty]$. The triangle inequality results from the fact that interleavings can be composed (adding the respective $\epsilon$-values). If $V,W$ are ob-isomorphic then $\inter(V,W) = 0$, so the `pseudo' is necessary.

\medskip
Associated to a persistence module~$V$ is its \emph{persistence measure}~\cite{Chazal_dS_G_O_2012}. This is a function defined on rectangles $[a,b] \times [c,d]$ by the formula
\begin{align*}
&\mu_V([a,b] \times [c,d])
\\
&\qquad =
\text{multiplicity of}\;\;
\big(
0 \longrightarrow \kk \longrightarrow \kk \longrightarrow 0
\big)
\;\;\text{in}\;\;
\big(
V_a \longrightarrow V_b \longrightarrow V_c \longrightarrow V_d
\big).
\end{align*}
The rectangle must lie in the extended half-plane
\[
\mathbb{H} =
\left\{
(p,q) \mid -\infty \leq p < q \leq +\infty
\right\},
\]
so $-\infty \leq a < b < c < d \leq +\infty$.
We set $V_{-\infty} = V_{+\infty} = 0$ to interpret the extreme cases.
The measure is additive with respect to splitting a rectangle into smaller rectangles, and therefore (being nonnegative) it is monotone with respect to inclusions of rectangles.

\medskip
The \emph{undecorated diagram} $\dgm(V)$ of a persistence module~$V$ is a multiset in~$\mathbb{H}$ defined, following~\cite{Chazal_dS_G_O_2012}, by the multiplicity function
\[
\mult_V(p,q)
=
\min
\left(
\mu_V([a,b] \times [c,d])
\mid
a < p < b < c < q < d
\right).
\]
We allow $-\infty < -\infty$ and $+\infty < +\infty$ when selecting $a$ and~$d$. Because of monotonicity, the minimum can be interpreted as a limit over a decreasing sequence of rectangles which contain $(p,q)$ in their interior.
The set of values $(p,q)$ of finite multiplicity is an open subset of $\mathbb{F}_V \subseteq \mathbb{H}$, the \emph{finite support} of~$V$. Within the finite support, the undecorated diagram is locally finite.
It is known that:
\begin{vlist}
\item
If $V$ is q-tame then $\mathbb{F}_V = \mathbb{H}$.
\item
If $V$ is q-tame and decomposable into intervals, then the undecorated diagram records exactly the endpoints of the intervals in the decomposition.
\end{vlist}
There is also a decorated diagram which is capable of distinguishing open, closed and half-open intervals.

\medskip
Two diagrams $\dgm(V)$, $\dgm(W)$ may be compared using the \emph{bottleneck distance}. Let ${\sim}$ denote a partial matching between the points of $\dgm(V)$ and $\dgm(W)$ in the respective finite supports. The cost of the partial matching is
\[
\mathrm{cost}(\sim)
=
\sup
\begin{cases}
\ellinf(v,w)
\quad &
\text{matched pairs $v \sim w$}
\\
\ellinf(v, \mathbb{H} - \mathbb{F}_W)
\quad &
\text{unmatched $v$}
\\
\ellinf(w, \mathbb{H} - \mathbb{F}_V)
\quad &
\text{unmatched $w$}
\end{cases}
\]
where we use the metric $\ellinf((p_1,q_1),(p_2,q_2)) = \max(|p_1-p_2|,|q_1-q_2|)$.
The bottleneck distance between diagrams is defined as
\[
\bottle(\dgm(V),\dgm(W))
=
\inf( \mathrm{cost}(\sim) \mid \text{$\sim$ is a partial matching}).
\]
One can show, using a compactness argument, that the infimum is attained.

\begin{thm}
[stability and isometry~\cite{CohenSteiner_E_H_2007,Lesnick_2011,Chazal_dS_G_O_2012}]
For arbitrary persistence modules $V,W$ over the real line, we have
\[
\bottle(\dgm(V), \dgm(W)) \leq \inter(V,W).
\]
If $V,W$ are q-tame then equality holds.
\qed
\end{thm}

\subsection{Results in the observable category}

We now transport our discussion to the observable category.

\begin{pro}
\label{pro:inter-ob}
The interleaving distance between persistence modules is observable.
\end{pro}

\begin{proof}
We know that $\mathrm{d}_i(V,V') = 0$ whenever $V,V'$ are ob-isomorphic. If also $W,W'$ are ob-isomorphic, it follows that $\mathrm{d}_i(V',W') = \mathrm{d}_i(V,W)$ by the triangle inequality.
\end{proof}

\begin{pro}
\label{pro:dgm-ob}
The undecorated persistence diagram is observable.
\end{pro}

\begin{proof}
Let $\phi: V \wto W$ be an ob-isomorphism with inverse $\psi: W \wto V$. We will show that $\mult_V(p,q) = \mult_W(p,q)$ for all points $(p,q)$.
Let $a,b,c,d$ be values attaining the minimum in the definition of $\mult_V(p,q)$, and select $a', b', c', d'$ such that
\[
a < a' < p < b' < b < c < c' < q < d' < d.
\]
Thus $(p,q)$ lies in the interior of $[a',b'] \times [c',d']$ which lies in the interior of $[a,b] \times [c,d]$.
From the commutative diagram
\[
\begin{diagram}
\node{V_a}
	\arrow[3]{e,t}{\rho_{ba}}
	\arrow{se,b}{\phi_{a'a}}
\node[3]{V_b}
	\arrow{e,t}{\rho_{cb}}
\node{V_c}
	\arrow[3]{e,t}{\rho_{dc}}
	\arrow{se,b}{\phi_{c'c}}
\node[3]{V_d}
\\
\node[2]{W_{a'}}
	\arrow{e,b}{\sigma_{b'a'}}
\node{W_{b'}}
	\arrow[3]{e,b}{\sigma_{c'b'}}
	\arrow{ne,b}{\psi_{bb'}}
\node[3]{W_{c'}}
	\arrow{e,b}{\sigma_{d'c'}}
\node{W_{d'}}
	\arrow{ne,b}{\psi_{dd'}}
\end{diagram}
\]
it follows by applying monotonicity (to the eight-term chain of vector spaces) that $\mu_W([a',b'] \times [c',d']) \leq \mu_V([a,b] \times [c,d])$, and therefore $\mult_W(p,q) \leq \mult_V(p,q)$. The reverse inequality follows symmetrically.
\end{proof}

\begin{cor}
The stability and isometry theorem for persistence modules over the real line is meaningful and true in the observable category.
\qed
\end{cor}

There is a particularly clean structure theory for q-tame modules in $\Ob$.

\begin{thm}
\label{thm:real-main}
Let $V,W$ be q-tame persistence modules over the real line. The following statements are equivalent:
\begin{vlist}
\item[(a)]
$V$ and $W$ are ob-isomorphic.

\item[(b)]
The interleaving distance between $V$ and $W$ is zero.

\item[(c)]
The undecorated persistence diagrams of $V$ and $W$ are equal.

\end{vlist}
\end{thm}

\begin{proof}
We have seen (a)$\,\Rightarrow\,$(b).

(b)$\,\Rightarrow\,$(c):
The stability theorem implies that the bottleneck distance between the diagrams is zero. Since q-tame persistence modules have locally finite diagrams, it follows that the diagrams are equal.

(c)$\,\Rightarrow\,$(a):
Being q-tame, the modules $V,W$ are ob-isomorphic to direct sums of interval modules. We may assume that the intervals are open and nonempty; then the intervals are determined by the persistence diagrams, so the two direct sums are isomorphic.
\end{proof}

We finish by showing what happens when we drop q-tameness.

\begin{exa}
We construct a pair of persistence modules $V,W$ whose interleaving distance is zero but which are not ob-isomorphic. Let $K$ be a compact subset of the half-plane with no isolated points, and let $X,Y$ be countable dense subsets of~$K$. If $X \not= Y$ then
\[
V = \bigoplus_{(p,q) \in X} \kk_{(p,q)}
\quad
\text{and}
\quad
W = \bigoplus_{(p,q) \in Y} \kk_{(p,q)}
\]
are not ob-isomorphic, by Theorem~\ref{thm:KSA}.
Now let $\epsilon > 0$. Select a bijection $f: X \to Y$ that moves points by at most~$\epsilon$. Each matched pair of summands $\kk_{(p,q)}$, $\kk_{f(p,q)}$ is $\epsilon$-interleaved, so $V, W$ are $\epsilon$-interleaved. Thus the interleaving distance between $V$ and~$W$ is zero.
\qed
\end{exa}

\begin{exa}
We construct a persistence module~$V$ indexed by~$\R$ which is not ob-isomorphic to a direct sum of interval modules. Let $\hat{V}$ be a persistence module indexed by~$\Z$ that is not isomorphic to a direct sum of interval modules (such as the example of Webb~\cite{Webb_1985}). Define $V$ by setting $V_t = \hat{V}_{\lfloor t \rfloor}$ and $\rho_{ts} = \rho_{\lfloor t \rfloor \lfloor s \rfloor}$.

Certainly $V$ cannot decompose into interval modules because that would induce an interval decomposition of~$\hat{V}$. We show that the same is true for any module~$W$ ob-isomorphic to~$V$. 
To show this, let $\hat{W}$ be the module indexed by~$\Z$ defined by
\[
\hat{W}_n = \Ima(W_{n+(1/5)} \to W_{n+(3/5)})
\]
with structure maps induced by those of~$W$. Then any direct-sum decomposition of~$W$ induces a direct sum decomposition of~$\hat{W}$, and interval module summands of~$W$ become interval module summands of~$\hat{W}$.
Meanwhile, thanks to the ob-isomorphism between $V,W$ we have a commutative diagram:
\[
\begin{diagram}
\dgARROWLENGTH=2em
\node{V_{n}}
	\arrow[2]{e}
	\arrow{se}
\node[2]{V_{n+(2/5)}}
	\arrow[2]{e}
	\arrow{se,t}{\!\!*_n}
\node[2]{V_{n+(4/5)}}
\\
\node[2]{W_{n+(1/5)}}
	\arrow[2]{e}
	\arrow{ne}
\node[2]{W_{n+(3/5)}}
	\arrow{ne}
\end{diagram}
\]
The top row is just $\hat{V}_n = \hat{V}_n = \hat{V}_n$, and it follows that the map labelled~$*_n$ induces an isomorphism between $\hat{V}_n$ and $\hat{W}_n$. From the diagram
\[
\begin{diagram}
\dgARROWLENGTH=2em
\node{V_{m+(2/5)}}
	\arrow[2]{e}
	\arrow{se,t}{\!\!*_m}
\node[2]{V_{n+(2/5)}}
	\arrow{se,t}{\!\!*_n}
\\
\node[2]{W_{m+(3/5)}}
	\arrow[2]{e}
\node[2]{W_{n+(3/5)}}
\end{diagram}
\]
we see that the structure maps agree under these isomorphisms. We conclude that $\hat{V}, \hat{W}$ are isomorphic. An interval decomposition of~$W$ would induce an interval decomposition of~$\hat{V}$ which, by assumption, does not exist.
\qed
\end{exa}

\section*{Acknowledgements}

The authors wish to thank Michael Lesnick for many stimulating questions and conversations that inspired and helped shape this work.

The first author acknowledges the support of the French ANR project TopData ANR-13-BS01-0008.
The third author gratefully acknowledges the support of the Simons Foundation (grant \#267571). He thanks his home institution, Pomona College, for a sabbatical leave of absence in 2013--14. The sabbatical was hosted by the Institute for Mathematics and its Applications, University of Minnesota, with funds provided by the National Science Foundation.

\bibliographystyle{plain}
\bibliography{pQuotient}

\end{document}